\documentclass{amsart}
\usepackage{mathrsfs,latexsym,amsfonts,amssymb}

\newtheorem{theorem}{Theorem}[section]
\newtheorem{lemma}[theorem]{Lemma}
\newtheorem{corollary}[theorem]{Corollary}
\newtheorem{question}[theorem]{Question}
\theoremstyle{definition}
\newtheorem{definition}[theorem]{Definition}
\newtheorem{proposition}[theorem]{Proposition}
\theoremstyle{remark}

\begin{document}

\title[A note on rectifiable spaces]
{A note on rectifiable spaces}

\author{Fucai Lin}
\address{(Fucai Lin): Department of Mathematics and Information Science,
Zhangzhou Normal University, Zhangzhou 363000, P. R. China}
\email{linfucai2008@yahoo.com.cn}
\author{Chuan Liu*}
\address{(Chuan Liu): Department of Mathematics,
Ohio University Zanesville Campus, Zanesville, OH 43701, USA}
\email{liuc1@ohio.edu}
\author{Shou Lin}
\address{(Shou Lin): Institute of Mathematics, Ningde Teachers' College, Ningde, Fujian
352100, P. R. China} \email{linshou@public.ndptt.fj.cn}

\thanks{Supported by the NSFC (No. 10971185) and the Educational Department of Fujian Province (No. JA09166) of China.\\
* corresponding author}

%insert keywords
%insert subject class

%\date{\today}
\begin{abstract}
In this paper, we firstly discuss the question: Is $l_{2}^{\infty}$ homeomorphic to a rectifiable space or a paratopological group? And then, we mainly discuss locally compact rectifiable spaces, and show that a locally compact rectifiable space with
the Souslin property is $\sigma$-compact, which gives an affirmative answer to A.V. Arhangel'ski\v{i}  and M.M. Choban's question [On remainders of rectifiable spaces, Topology Appl.,
157(2010), 789-799]. Next, we show that a rectifiable space $X$ is strongly Fr$\acute{e}$chet-Urysohn if and only if $X$ is an $\alpha_{4}$-sequential space. Moreover, we discuss the metrizabilities of rectifiable spaces, which gives a partial answer for a question posed in \cite{LFC2009}. Finally, we consider the remainders of  rectifiable spaces, which improve some results in \cite{A2005, A2007, A2009, Liu2009}.
\end{abstract}

\maketitle

\section{Introduction}
Recall that a {\it topological group} $G$ is a group $G$ with a
(Hausdorff) topology such that the product maps of $G \times G$ into
$G$ is jointly continuous and the inverse map of $G$ onto itself
associating $x^{-1}$ with arbitrary $x\in G$ is continuous. A {\it
paratopological group} $G$ is a group $G$ with a topology such that
the product maps of $G \times G$ into $G$ is jointly continuous. A
topological space $G$ is said to be a {\it rectifiable space} \cite{C1987}
provided that there are a surjective homeomorphism $\varphi :G\times
G\rightarrow G\times G$ and an element $e\in G$ such that
$\pi_{1}\circ \varphi =\pi_{1}$ and for every $x\in G$ we have
$\varphi (x, x)=(x, e)$, where $\pi_{1}: G\times G\rightarrow G$ is
the projection to the first coordinate. If $G$ is a rectifiable
space, then $\varphi$ is called a {\it rectification} on $G$. It is
well known that rectifiable spaces and paratopological groups are
all good generalizations of topological groups. In fact, for a
topological group with the neutral element $e$, then it is easy to
see that the map $\varphi (x, y)=(x, x^{-1}y)$ is a rectification on
$G$. However, there exists a paratopological group which is not a
rectifiable space; Sorgenfrey line (\cite[Example
1.2.2]{E1989}) is such an example. Also, the 7-dimensional sphere $S_{7}$ is
rectifiable but not a topological group \cite[$\S$ 3]{V1990}.
Further, it is easy to see that paratopological groups and
rectifiable spaces are all homogeneous.

By a remainder of a space $X$ we understand the subspace
$bX\setminus X$ of a Hausdorff compactification $bX$ of $X$.

In section 3, we show that $l_{2}^{\infty}$ is homeomorphic to no rectifiable space or paratopological group, where $l_{2}^{\infty}$ is the separable Hilbert space, which extends a result of T. Banakh in \cite{BT}. In section 4, we mainly discuss locally compact rectifiable spaces, and show that a locally compact and separable rectifiable space is $\sigma$-compact, which give an affirmative answer for a question of A.V. Arhangel'ski\v{i}  and M.M. Choban's. Moreover, we prove that under the set theory assumption a locally compact rectifiable space with the $\alpha_{4}$-properties is metrizable.
In section 5, we show that a rectifiable space $X$ is strongly Fr$\acute{e}$chet-Urysohn if and only if $X$ is an $\alpha_{4}$-sequential space. In section 6, we mainly discuss the metrizability of rectifiable spaces which have a point-countable $k$-network. In section 7, we mainly consider the question: When does a
Tychonoff rectifiable space $G$ have a
Hausdorff compactification $bG$ with a remainder belonging to the class of separable and metrizable spaces?
\maketitle

\section{Preliminaries}
In \cite{PE}, E. Pentsak
studied the topology of the direct limit $X^{\infty}=\underset{\longrightarrow}{\lim}X^{n}$ of the sequence
$$X\subset X\times X\subset X\times X\times X\subset\cdots,$$
where $(X, \star)$ was a ``nice'' pointed space and $X^{n}$ was identified with the subspace $X^{n}\times \{\star\}$ of $X^{n+1}$.

A space $X$ is called an {\it $S_{2}$}-{space} ({\it Arens' space})  if
$X=\{\infty\}\cup \{x_{n}: n\in \mathbb{N}\}\cup\{x_{n}(m): m, n\in
\mathbb{N}\}$ and the topology is defined as follows: Each
$x_{n}(m)$ is isolated; a basic neighborhood of $x_{n}$ is
$\{x_{n}\}\cup\{x_{n}(m): m>k\}$ for some $k\in \mathbb{N}$;
a basic neighborhood of $\infty$ is $\{\infty\}\cup (\bigcup\{V_{n}:
n>k\})$ for some $k\in \mathbb{N}$, where $V_{n}$ is a
neighborhood of $x_{n}$.

T. Banakh define the space $K$ \cite{BT}.

Let $$K=\{(0, 0)\}\cup\{(\frac{1}{n}, \frac{1}{nm}): n,
m\in\mathbb{N}\}\subset \mathbb{R}^{2}.$$ The space $K$ is non-locally
compact and metrizable. Also, the space $K$ is a minimal space with
these properties in the sense that each metrizable non-locally
compact space contains a closed copy of $K$. For convenience, put
$x_{0}=(0, 0)$ and $x_{n, m}=(\frac{1}{n}, \frac{1}{nm})$ for any
$n, m\in\mathbb{N}$.

The space $X$ is called $S_{\omega}$ if $X$ is obtained by identifying all the
limit points of $\omega$ many convergent sequences.

If $A$ is a subset of a space $X$, then $[A]^{Seq}$
denotes the {\it sequential closure} of $A$, i.e. the set of limits of convergent sequences in
$A$. Clearly, we have $A\subset [A]^{seq}$. By induction on $\alpha\in\omega_{1}+1$, we can define $[A]_{\alpha}$ as follows:
$[A]_{0}=A, [A]_{\alpha+1}=[[A]_{\alpha}]^{seq}$ and $[A]_{\alpha}=\cup\{[A]_{\beta}: \beta<\alpha \}$ for a limit order $\alpha$. One
can easily verify that $[A]_{\omega_{1}+1}=[A]_{\omega_{1}}$, and that a space $X$ is sequential iff
$\overline{A}=[A]_{\omega_{1}}$ for every $A\subset X$. For a sequential space $X$ we define $so(X)$, the {\it sequential
order} of $X$, by $so(X)=\min\{\alpha\in\omega_{1}+1: \overline{A}=[A]_{\alpha}\ \mbox{for each}\ A\subset X\}$.

\begin{definition}
A space $X$ is said to be {\it Fr$\acute{e}$chet-Urysohn} if, for
each $x\in \overline{A}\subset X$, there exists a sequence
$\{x_{n}\}$ such that $\{x_{n}\}$ converges to $x$ and $\{x_{n}:
n\in\mathbb{N}\}\subset A$. A space $X$ is said to be {\it strongly
Fr$\acute{e}$chet-Urysohn} if the following condition is satisfied

(SFU) For every $x\in X$ and each sequence $\eta =\{A_{n}: n\in
\mathbb{N}\}$ of subsets of $X$ such that $x\in\bigcap_{n\in
\mathbb{N}}\overline{A_{n}}$, there is a sequence $\zeta =\{a_{n}:
n\in \mathbb{N\}}$ in $X$ converging to $x$ and intersecting
infinitely many members of $\eta$.
\end{definition}

Obviously, a strongly Fr$\acute{e}$chet-Urysohn space is
Fr$\acute{e}$chet-Urysohn. However, the space $S_{\omega}$ is
Fr$\acute{e}$chet-Urysohn and non-strongly
Fr$\acute{e}$chet-Urysohn.

Let $X$ be a space. For $P\subset X$, the set $P$ is a {\it sequential
neighborhood} of $x$ in $X$ if every sequence converging to $x$ is
eventually in $P$.

\begin{definition}
Let $\mathscr{P}=\bigcup_{x\in X}\mathscr{P}_{x}$ be a cover of a
space $X$ such that for each $x\in X$, (a) if $U,V\in
\mathscr{P}_{x}$, then $W\subset U\cap V$ for some $W\in
\mathscr{P}_{x}$; (b) the family $\mathscr{P}_{x}$ is a network of $x$ in $X$,
i.e., $x\in\bigcap\mathscr{P}_x$, and if $x\in U$ with $U$ open in
$X$, then $P\subset U$ for some $P\in\mathscr P_x$.

The family $\mathscr{P}$ is called a {\it weak base} for $X$ \cite{A1966} if, for every $G\subset X$, the set $G$ must be open in $X$ whenever for each $x\in G$ there exists $P\in
\mathscr{P}_{x}$ such that $P\subset G$.
The space $X$ is {\it weakly first-countable} if the family $\mathscr{P}$ is a weak base for $X$ such that each $\mathscr{P}_{x}$ is countable.
\end{definition}

The following theorem for the first time
there was announced in \cite{C1987}, and the readers can see the proof in \cite{C1992, G1996, V1989}.

\begin{theorem}\cite{C1987}\label{t9}
A topological space $G$ is rectifiable if and only if there exists $e\in G$ and two
continuous maps $p: G^{2}\rightarrow G$, $q: G^{2}\rightarrow G$
such that for any $x\in G, y\in G$ the next
identities hold:
$$p(x, q(x, y))=q(x, p(x, y))=y, q(x, x)=e.$$
\end{theorem}

In fact, we can assume that $p=\pi_{2}\circ \varphi^{-1}$ and
$q=\pi_{2}\circ \varphi$ in Theorem~\ref{t9}. Fixed a point $x\in
G$, then $f_{x}, g_{x}: G\rightarrow G$ defined with $f_{x}(y)=p(x,
y)$ and $g_{x}(y)=q(x, y)$, for each $y\in G$, are homeomorphism,
respectively. We denote $f_{x}, g_{x}$ with $p(x, G), q(x, G)$,
respectively.

Let $G$ be a rectifiable space, and let $p$ be the multiplication on
$G$. Further, we sometime write $x\cdot y$ instead of $p(x, y)$ and
$A\cdot B$ instead of $p(A, B)$ for any $A, B\subset G$. Therefore,
$q(x, y)$ is an element such that $x\cdot q(x, y)=y$; since $x\cdot
e=x\cdot q(x, x)=x$ and $x\cdot q(x, e)=e$, it follows that $e$ is a right neutral
element for $G$ and $q(x, e)$ is a right inverse for $x$. Hence a
rectifiable space $G$ is a topological algebraic system with
operation $p, q$, 0-ary operation $e$ and identities as above. It is
easy to see that this algebraic system need not to satisfy the
associative law about the multiplication operation $p$. Clearly,
every topological loop is rectifiable.

All spaces are $T_1$ and regular unless stated otherwise.
The notation $\mathbb{N}$ denotes the set of all positive natural numbers. The letter $e$
denotes the neutral element of a group and the right neutral element
of a rectifiable space, respectively. Readers may refer to
\cite{A2008, E1989, Gr1984} for notations and terminology not
explicitly given here.
\bigskip

\section{$l_{2}^{\infty}$ is homeomorphic to no rectifiable space or paratopological groups}
In this section, by a modification of the proof of Theorem 1 in \cite{BT}, we show that $l_{2}^{\infty}$ is homeomorphic to no rectifiable space or paratopological group.

We call a subset $A$ of a rectifiable space (resp. paratopological group) $G$ {\it multiplicative}
if for any $a, b\in A$ we have $p(a, b)=a\cdot b\in A$ (resp. $ab\in A$).

We denote by conv($K$)=$\{(0, 0)\}\cup\{(x, y): 0<y\leq x\leq 1\}$, where
conv($K$) is the convex hull of $K$ in $\mathbb{R}^{2}$.

In this section, we may assume that $S_{\omega}=\{y_{0}\}\cup\{y_{n, m}=(n,
\frac{1}{m}): n, m\in\mathbb{N}\}$, where, for each $n\in
\mathbb{N}$, the sequence $\{y_{n, m}\}\rightarrow y_{0}$ as
$m\rightarrow\infty$. For each $k\in \mathbb{N}$, let
$V_{k}=\{y_{0}\}\cup\{y_{n, m}: n\leq k, m\in \mathbb{N}\}$. It is easy to see that $S_{\omega}$ has the direct
limit topology with respect to the sequence $V_{1}, V_{2}, \cdots.$

\begin{theorem}\label{t0}
Let $X$ be a normal $k$-space. If $X$ contains closed copies of $S_{\omega}$ and
$K$, then it is homeomorphic to no closed multiplicative subset of a
rectifiable space $G$ such that $y_{0}$ is
the right neutral element of $G$.
\end{theorem}

\begin{proof}
Suppose not, let $X$ be a closed multiplicative subset of a
rectifiable space $G$. Now, we define a map $f: K\times
 S_{\omega}\rightarrow X$ with $f(x, y)=p(x, y)$ for each $(x, y)\in
K\times S_{\omega}$. Then the following (1) and
(2) hold:

(1) the map $(\pi_{1}, f): K\times S_{\omega}\rightarrow K\times X$
is a closed embedding, where $(\pi_{1}, f)(x, y)=(x, f(x, y))$ for
each $(x, y)\in K\times S_{\omega}$;

(2) the map $g: K\rightarrow X$ defined by $g(x)=p(x, y_{0})=x$ (that is, $g$ is the
identity map), for each $x\in K$, is a closed embedding.

Indeed, the statement (2) is obvious. Moreover, it is easy to see that the map $(\pi_{1}, f): K\times S_{\omega}\rightarrow K\times X$
is injective continuous. We only show that the map $(\pi_{1}, f)$ is relatively open. For each open subset $U\times V$ of $K\times S_{\omega}$, we have $(\pi_{1}, f)(U\times V)=\cup\{\{x\}\times (x\cdot V): x\in U\}$, where $U$ and $V$ are open in $K$ and $S_{\omega}$ respectively. Since $V$ is open in $S_{\omega}$, there exists an open subset $W$ of $X$ such that $W\cap S_{\omega}=V$. Therefore, we have $$(\pi_{1}, f)(U\times V)=\cup\{\{x\}\times p(x, V): x\in U\}=(U\times (U\cdot W))\cap (\pi_{1}, f)(K\times S_{\omega}).$$ Since $p(U, W)=\cup\{p(x, W): x\in U\}$ is open in $X$, the set $U\times p(U, W)$ is open in $K\times X.$

By the normality of $X$, let $h: X\rightarrow\mbox{conv}(K)$ be a
continuous extension of the map $g^{-1}: g(K)\rightarrow K$.

For each $n, m\in \mathbb{N}$, let $\delta_{n,
m}=\frac{1}{2nm(m+1)}$, and put
$$W_{n, m}=\mbox{conv}(K)\cap ((\frac{1}{n}-\delta_{n,
m}, \frac{1}{n}+\delta_{n, m})\times (\frac{1}{nm}-\delta_{n, m},
\frac{1}{nm}+\delta_{n, m})).$$ Obviously, the collection $\{W_{n, m}: n,
m\in\mathbb{N}\}$ consists of pairwise disjoint neighborhoods
of the points $x_{n, m}$ in conv($K$). Since $y_{n, m}\rightarrow
y_{0}$ as $m\rightarrow\infty$ and $h\circ f(x_{n, m}, y_{0})=h(p(x_{n, m}, y_{0}))=x_{n,
m}=(\frac{1}{n}, \frac{1}{nm})$, for any $n, m\in \mathbb{N}$, there
exists a $k(n, m)\in \mathbb{N}$ such that $h\circ f(x_{n, m},
y_{n, k(n, m)})\in W_{n, m}$. Without loss of generality, we may assume
that $k(n, m+1)>k(n, m)$ for any $n, m\in \mathbb{N}$. Put
$$Z=\{p(x_{n, m}, y_{n, k(n, m)}): n, m\in \mathbb{N}\}.$$ For each $n, m\in \mathbb{N}$, it
follows from $h\circ f(x_{0}, y_{0})\not\in W_{n, m}$ that $f(x_{0},
y_{0})\not\in Z$.

Claim: $Z$ is closed in $X$.

Since $X$ is a $k$-space, it suffices to prove that for each compact
subset $F$ of $X$ the intersection $F\cap Z$ is closed in $F$. Let
$$F_{1}=\{x_{0}\}\cup\{x_{n, m}:
h(F)\cap W_{n, m}\neq\emptyset , n, m\in \mathbb{N}\}\ \mbox{and}\
F_{2}=\pi_{2}((\pi_{1}, f)^{-1}(F_{1}\times F)).$$ Since
$h(F)\subset \mbox{conv}(K)$ is compact, the set $F_{1}$ is compact. It
follows from (1) that $$(\pi_{1}, f)^{-1}(F\times F_{1})\subset
K\times S_{\omega}$$ is compact, and hence $F_{2}$ is also compact.
Because $S_{\omega}=\underset{\longrightarrow}{\lim}V_{n}$, there exists an
$n_{0}\in \mathbb{N}$ such that $F_{2}\subset V_{n_{0}}$. Since
$F\cap Z\subset p(F_{1}, F_{2})$, we have $F\cap Z\subset \{p(x_{n, m}, y_{n, k(n,
m)}): n\leq n_{0}, y_{n, m}\in F_{1}\}$. By the
compactness of $F_{1}$, it is easy to see that $\{p(x_{n,
m}, y_{n, k(n, m)}): n\leq n_{0}, x_{n, m}\in F_{1}\}$ is finite. Therefore, the set $F\cap
Z$ is closed in $F$.

Since $p(x_{0}, y_{0})=x_{0}\not\in Z$ and $p$ is continuous, it follows
from the Claim that there exist open neighborhoods $V(x_{0})\subset K$ and $U(y_{0})\subset
S_{\omega}$ of $x_{0}$ and $y_{0}$ respectively such that
$p(V(x_{0}), U(y_{0}))\cap Z=\emptyset$. For every $m\in
\mathbb{N}$, we can fix an $n\in \mathbb{N}$ such that $x_{n, m}\in
V(x_{0})$. Since $\{y_{n ,m}\}_{m=1}^{\infty}$ converges to $y_{0}$
and $\{k(n, m)\}_{m=1}^{\infty}$ is increasing, there is an $m\in
\mathbb{N}$ such that $y_{n, k(n, m)}\in U(y_{0}).$ Then $p(x_{n, m}, y_{n,
k(n, m)})\in p(V(x_{0}), U(y_{0}))\cap Z$, which is a
contradiction.
\end{proof}

\begin{theorem}\label{t2}
Let $X$ be a normal $k$-space. If $X$  contains closed copies of $S_{\omega}$ and
$K$, then it is homeomorphic to no closed multiplicative subset of a
paratopological group.
\end{theorem}

\begin{proof}
Suppose not, let $X$ be a closed multiplicative subset of a
rectifiable space $G$. Now, we define a map $f: K\times
 S_{\omega}\rightarrow X$ with $f(x, y)=xy$ for each $(x, y)\in
K\times S_{\omega}$. Obviously, we can obtain the following (1) and
(2):

(1) the map $(\pi_{1}, f): K\times S_{\omega}\rightarrow K\times X$
is a closed embedding, where $(\pi_{1}, f)(x, y)=(x, xy)$ for
each $(x, y)\in K\times S_{\omega}$;

(2) the map $g: K\rightarrow X$ defined by $g(x)=xy_{0}$, for each $x\in K$, is a closed embedding.

By the normality of $X$, let $h: X\rightarrow\mbox{conv}(K)$ be a
continuous extension of the map $g^{-1}: g(K)\rightarrow K$.

By the proof of Theorem~\ref{t0}, we can define the neighborhoods $W_{n, m}$ of the points $x_{n, m}$ in conv($K$) and the closed set $Z$ with
$x_{0}y_{0}\not\in Z$.

Since $f(x_{0},
y_{0})=x_{0}y_{0}\not\in Z$, $G$ is joint continuous and $Z$ is closed, it follows
that there exist open neighborhoods $V(x_{0})\subset K$ and $U(y_{0})\subset
S_{\omega}$ of $x_{0}$ and $y_{0}$ respectively such that
$(V(x_{0})\times U(y_{0}))\cap Z=\emptyset$. For every $m\in
\mathbb{N}$, we can fix an $n\in \mathbb{N}$ such that $x_{n, m}\in
V(x_{0})$. Since $\{y_{n ,m}\}_{m=1}^{\infty}$ converges to $y_{0}$
and $\{k(n, m)\}_{m=1}^{\infty}$ is increasing, there is an $m\in
\mathbb{N}$ such that $y_{n, k(n, m)}\in U(y_{0}).$ Then $x_{n, m}y_{n,
k(n, m)}\in (V(x_{0})\times U(y_{0}))\cap Z$, which is a
contradiction.
\end{proof}

It is well known that a space $X$ contains a closed copy of $S_{\omega}$, provided $X$ can be written as a direct
limit of a sequence
$$X_{1}\subset X_{2}\subset\cdots,$$where each $X_{n}$ is a closed metrizable subset of $X$, nowhere dense in $X_{n+1}$. In
particular, the space $l_{2}^{\infty}$ contains a topological closed copy of $S_{\omega}$. Moreover, the space $l_{2}^{\infty}$ is a normal $k$-space and contains a topological closed copy of $K$.
Therefore, by the topological homogeneity of $l_{2}^{\infty}$ and Theorems~\ref{t0} and~\ref{t2}, we have the following theorem.

\begin{theorem}
$l_{2}^{\infty}$ is homeomorphic to no rectifiable space or a paratopological group.
\end{theorem}

\begin{corollary}
$l_{2}^{\infty}$ is homeomorphic to no topological loop.
\end{corollary}

\begin{corollary}\cite{BT}
$l_{2}^{\infty}$ is homeomorphic to no topological group.
\end{corollary}
\bigskip

\section{locally compact rectifiable spaces}

In \cite{A2009}, A.V. Arhangel'ski\v\i \ and M.M. Choban posed the following question:

\begin{question}\cite[Problem 5.10]{A2009}\label{q3}
Is every rectifiable $p$-space with a countable Souslin number
Lindel$\ddot{\mbox{o}}$f? What if we assume the space to be
separable? Separable and locally compact?
\end{question}

Now, we give an affirmative answer for Questions~\ref{q3} of the case of separable and locally compact rectifiable spaces.

\begin{lemma}\label{l15}
Let $G$ be a rectifiable space. If $Y$ is a dense subset of $G$ and $U$ is an open neighborhood of the right neutral element $e$ of $G$, then $G=Y\cdot U$.
\end{lemma}

\begin{proof}
Fix an arbitrary $z\in G$. Since $q(z, z)=e\in U$, there exists an open neighborhood $V$ of $e$ such that $q(z\cdot V, z)\subset U$. Put $W=z\cdot V$. Then $W$ is an open neighborhood of $z$ in $G$. Since $Y$ is a dense subset of $G$, we have $W\cap Y\neq\emptyset$. Take a point $y\in W\cap Y=z\cdot V\cap Y$. Then  $y=z\cdot v$ for some $v\in V$. $$z=p(z\cdot v, q(z\cdot v, z))=p(y, q(z\cdot v, z))\in p(y, q(z\cdot V, z))\subset p(y, U)=y\cdot U\subset Y\cdot U.$$

By the choice of arbitrary point of $z$, we have $G=Y\cdot U$.
\end{proof}

It follows from Lemma~\ref{l15}, we have the following results, which give an answer for Question~\ref{q3}.

\begin{theorem}
If $G$ is a locally $\sigma$-compact\footnote{A space $X$ is {\it locally $\sigma$-compact} if, for each point $x$ of $X$, there exists an open neighborhood $U_{x}$ of $x$ such that $U_{x}$ can be cover by a countably many compact subsets of $X$.} rectifiable space with
the Souslin property, then $G$ is $\sigma$-compact.

\end{theorem}

\begin{proof}
Let $G$ be a locally $\sigma$-compact rectifiable space with
the Souslin property. For each $\alpha\in \Gamma$, let $\mathcal{A}_\alpha$ be the family consisting of disjoint open subsets of $G$ such that each element of $\mathcal{A}_\alpha$ is covered by countably many compact subsets (since $G$ is locally $\sigma$-compact).  $\{\mathcal{A}_\alpha, \alpha\in \Gamma\}$ is a set with partial order by inclusion.
It is easy to see that every chain of $\{\mathcal{A}_\alpha, \alpha\in \Gamma\}$ has an upper bound, by Zorn's Lemma, there is a maximal element $\mathcal{A} \in \{\mathcal{A}_\alpha, \alpha\in \Gamma\}$. Since $G$ has Souslin property, we have $|\mathcal{A}|\leq \omega$, and hence we write $\mathcal{A}=\{A_i\}$, $A_i\subset \cup\{K_{i, j}\}$, where each $K_{i, j}$ is a compact subset of $G$. By maximality of $\mathcal{A}$, $\cup \{K_{i, j}\}$ is a dense subset of $G$. Let $U$ be an open neighborhood of $e$, which is covered by countably many compact subsets $\{H_l\}$. By Lemma ~\ref{l15}, $G=(\cup \{K_{i, j}\})\cdot (\cup\{H_l\})=\cup (K_{i, j}\cdot H_l)$, each $K_{i, j}\cdot H_l$ is compact, hence $G$ is $\sigma$-compact.
\end{proof}

\begin{corollary}
If $G$ is a locally compact and separable rectifiable space, then $G$ is $\sigma$-compact (and, hence, Lindel\"{o}f).
\end{corollary}

\begin{corollary}
If $G$ is a locally Lindel\"{o}f and separable rectifiable space, then $G$ is Lindel\"{o}f.
\end{corollary}

\smallskip

Let $A$ be a subspace of a rectifiable space $G$. Then $A$ is called {\it a rectifiable subspace of $G$} if we have $p(A, A)\subset A$ and $q(A, A)\subset A$.

\begin{proposition}\label{p1}
Let $G$ be a rectifiable space. If $H$ is a rectifiable subspace of $G$, then $\overline{H}$ is also a rectifiable subspace of $G$.
\end{proposition}

\begin{proof}
Take two points $x, y\in\overline{H}$. Then we shall show that $p(x, y)\in\overline{H}$ and $q(x, y)\in\overline{H}$.

Since $x, y\in \overline{H}$, there exist two nets $\{x_\alpha\}, \{y_\beta\}$ in $H$ such that $x_\alpha \to x, y_\beta \to y$. Since $p$ is continuous, $p(x, y)$ is a cluster point of $\{p(x_\alpha, y_\beta\}\subset H$. Hence $p(x, y)\in \overline{H}$.

Similarly, we can show that $q(x, y)\in\overline{H}$.
\end{proof}

\begin{lemma}\label{l6}
Let $G$ be a rectifiable space. If $V$ is an open rectifiable subspace of $G$, then $V$ is closed in $G$.
\end{lemma}

\begin{proof}
Suppose that $V$ is non-closed in $G$. Then $\overline{V}\setminus V\neq\emptyset.$ Take a point $x\in\overline{V}\setminus V$. Since $q(x, x)=e\in V$ and the continuity of $q$, there exists an open neighborhood $W$ of $e$ such that $q(x\cdot W, x)\subset V$. Put $U=x\cdot W$. Then $U$ is an open neighborhood of $x$, and hence $U\cap V\neq\emptyset$ since $x\in\overline{V}$. Therefore, there exist $a\in W$ and $b\in V$ such that $x\cdot a=b.$ Then we have $$x=p(x\cdot a, q(x\cdot a, x))=p(b, q(x\cdot a, x))\subset p(V, V)=V,$$ where $p(V, V)=V$ since $V$ is a rectifiable subspace of $G$. However, the point $x\not\in V$, which is a contradiction.
\end{proof}

\begin{theorem}\label{t4}
If $H$ is a locally compact rectifiable subspace of a rectifiable space $G$, then $H$ is closed in $G$.
\end{theorem}

\begin{proof}
Let $K=\overline{H}$. Then $K$ is a rectifiable subspace of $G$ by Proposition~\ref{p1}. Since $H$ is a dense locally compact subspace of $K$, it follows from \cite[Theorem 3.3.9]{E1989} that $H$ is open in $K$. By Lemma~\ref{l6}, the set $H$ is closed, and hence $K=H$.
\end{proof}

The following Lemma maybe was proved somewhere.

\begin{lemma}\label{compact-character}
Let $F$ be a compact subset of a space $X$ and have a countable base $\{U_n\}$ with $\overline{U_{n+1}}\subset U_n$ in $X$, and let $H=\cap_nV_n$ ($V_{n+1}\subset V_n$ and each $V_n$ is open in $F$) is a compact $G_\delta$-set of $F$. For $n\in \mathbb{N}$, let $W_n$ be an open set in $X$ such that $V_n = W_n\cap F$, $W_n\subset U_n$, $\overline{W_{n+1}}\subset W_n$, then $\{W_n\}$ is a countable base at $H$ in $X$.

\end{lemma}

\begin{proof}
$H=\cap_n W_n=\cap_n \overline{W_n}$. Suppose that $\{W_n\}$ is not a countable base at $H$, then there is an open subset $U$ of $X$ such that $H\subset U$ and $W_n\setminus U\neq \emptyset$. By induction, choose $x_n\in W_n\setminus U$ with $x_i\neq x_j$ if $i\neq j$. Since $x_n\in U_n$ for each $n\in \mathbb{N}$, then $\{x_n\}$ has a cluster point $x$. In fact, if $\{x_n\}\cap F$ is infinite, then $\{x_n\}$ has a cluster point in $F$ since $F$ is compact; if $\{x_n\}\cap F$ is finite, without loss generality, we assume  $\{x_n\}\cap F=\emptyset$. Since $F\subset X\setminus \{x_n\}$ which is open in $X$, there is $n_0\in\mathbb{N}$ such that $F\subset U_n\subset X\setminus \{x_n\}$ for $n>n_0$. This is a contradiction since $x_n\in U_n$. Therefore, we have $x\in \overline{W_n}$ for each $n$, then $x\in H\subset U$, and hence $U$ contains infinitely many $x_n's$, which is a contradiction.

\end{proof}

Next, we shall show that, for each locally compact rectifiable space, there exists a compact rectifiable subspace with a countable character.

\begin{lemma}\label{l17}
Let $G$ be a rectifiable space and $F$ be a compact subset of $G$ containing $e$ and having a countable base $\{U_{n}: n\in\mathbb{N}\}$ in $G$. Assume that a sequence $\zeta=\{V_{n}: n\in\mathbb{N}\}$ of open neighborhoods of $e$ in $G$ such that $\overline{V_{n+1}}\cdot\overline{V_{n+1}}\subset V_{n}\cap U_{n}$ and $q(V_{n+1}, V_{n+1})\subset V_{n}$. Then $H=\bigcap_{n\in\mathbb{N}}V_{n}$ is a compact rectifiable subspace of $G$, $H\subset F$ and $\zeta$ is a base for $G$ at $H$.
\end{lemma}

\begin{proof}
Obviously, we have $\overline{V_{n+1}}\subset V_{n}$ for each $n\in \mathbb{N}$. We first claim that $H$ is a compact rectifiable subspace of $G$.

Indeed, it is easy to see that $H=\bigcap_{n\in\mathbb{N}}V_{n}=\bigcap_{n\in\mathbb{N}}\overline{V_{n}}$, and hence $H$ is closed in $G$. For each $x, y\in H$, we have $x, y\in V_{n}$ for each $n\in \mathbb{N}$. Then, for each $n\in \mathbb{N}$,  we have $x\cdot y\in V_{n}$ since $\overline{V_{n+1}}\cdot\overline{V_{n+1}}\subset V_{n}$. Therefore, $x\cdot y\in H$. Since $q(V_{n+1}, V_{n+1})\subset V_{n}$, we have $q(x, y)\in H$. Therefore, $H$ is a rectifiable closed subspace. Obviously, $H\subset \bigcap_{n\in\mathbb{N}}U_{n}=F$. Thus $H$ is compact. By Lemma ~\ref{compact-character}, $\zeta$ is a base for $G$ at $H$.

\end{proof}

\begin{proposition}\label{p2}
Let $G$ be a rectifiable space with point-countable type. If $O$ is an open neighborhood of $e$, then there exists a compact rectifiable subspace $H$ of countable character in $G$ satisfying $H\subset O$.
\end{proposition}

\begin{proof}
Since $G$ is of point-countable type, there exists a compact subset of $G$ having a countable base in $G$.
By the homogeneity of $G$, we may assume that $e\in F$. Let $\{U_{n}: n\in\mathbb{N}\}$ be a countable base for $G$ at $F$. We define by induction a sequence $\{V_{n}: n\in\mathbb{N}\}$ of open neighborhoods of $e$ in $G$ satisfying the following conditions:

(1) $V_{1}\subset O;$

(2) $\overline{V_{n+1}}\cdot\overline{V_{n+1}}\subset V_{n}\cap U_{n}$ for each $n\in \mathbb{N}$;

(3) $q(V_{n+1}, V_{n+1})\subset V_{n}$ for each $n\in \mathbb{N}$.

 Put $H=\bigcap_{n\in \mathbb{N}}V_{n}$. Then $H\subset O$ by (1). It follows from Lemma~\ref{l17} that $H$ is a compact rectifiable subspace of $G$ and that $\{V_{n}: n\in\mathbb{N}\}$ is a base of $G$ at $H$.
\end{proof}

Since each locally compact space is of point-countable type, we have the following corollary.

\begin{corollary}
Let $G$ be a locally compact rectifiable space. If $O$ is an open neighborhood of $e$, then there exists a compact rectifiable subspace $H$ of countable character in $G$ satisfying $H\subset O$.
\end{corollary}

\begin{definition}
Let $X$ be a topological space. For $i=1, 4$ we say that $X$ is an {\it $\alpha_{i}$-space} if for each countable family $\{S_{n}: n\in\mathbb{N}\}$ of sequences converging to some point $x\in X$ there is a sequence $S$ converging to $x$ such that:

($\alpha_{1}$) $S_{n}\setminus S$ is finite for all $n\in \mathbb{N}$;

($\alpha_{4}$) $S_{n}\cap S\neq\emptyset$ for infinitely many $n\in \mathbb{N}$.
\end{definition}

Obviously, we have $\alpha_{1}\Rightarrow\alpha_{4}$.

Let $\omega^{\omega}$ denote the family of all functions from $\mathbb{N}$ into $\mathbb{N}$. For $f, g\in\omega^{\omega}$ we write $f<^{\ast}g$ if $f(n)<g(n)$ for all but finitely many $n\in \mathbb{N}$. A family $\mathscr{F}$ is {\it bounded} if there is a $g\in\omega^{\omega}$ such that $f<^{\ast}g$ for all $f\in\mathscr{F}$, and is {\it unbounded} otherwise. We denote by $\flat$ the smallest cardinality of an unbounded family in $\omega^{\omega}$. It is easy to see that $\omega <\flat\leq \mathrm{c}$, where $\mathrm{c}$ denotes the cardinality of the continuum.

\begin{lemma}\label{l16}\cite{NT}
For $i\in\{1, 4\}$, $D^{\tau}$ is an $\alpha_{i}$-space if and only if $\tau <\flat$, where $D$ is the discrete two-points space $\{0, 1\}$.
\end{lemma}

\begin{theorem} The following conditions are equivalent:
\begin{enumerate}
\item Every compact rectifiable space with the $\alpha_{1}$-property is metrizable;

\item Every  locally compact rectifiable space with the $\alpha_{4}$-property is metrizable;

\item $\flat=\omega_{1}$.
\end{enumerate}
\end{theorem}

\begin{proof}
The implication $(2)\Rightarrow (1)$ is trivial.

$(1)\Rightarrow (3)$. Since $D^{\omega_{1}}$ is a nonmetirzable compact group, so it cannot be an $\alpha_{1}$-space by (1). It follows from Lemma~\ref{l16} that $\flat\leq\omega_{1}$. Since $\flat >\omega$, it follows that $\flat=\omega_{1}$.

$(3)\Rightarrow (2)$. Suppose that $\flat=\omega_{1}$, and that $G$ is a locally compact $\alpha_{4}$-rectifiable space. Next, we shall prove that $G$ is metrizable. By Proposition~\ref{p2}, there exists a compact rectifiable subspace $F$ of $G$ which has a countable character at $F$ in $G$. We claim that $F$ is metrizable. If not, as proved V.V. Uspenskij in \cite{V1989, V1990}, compact rectifiable spaces are dyadic, and hence the space $F$ contains a subspace homeomorphic to $D^{\omega_{1}}$. Since a subspace of an $\alpha_{4}$-space is an $\alpha_{4}$-space, the subspace $D^{\omega_{1}}$ is an $\alpha_{4}$-space. Then, it follows from Lemma~\ref{l16} that $\omega_{1}<\flat$, which is a contradiction. Therefore, the space $F$ is metrizable.

Let $\{U_{n}: n\in\mathbb{N}\}$ be a countable base of $G$ at $F$, where $\overline{U_{n+1}}\subset U_{n}$ for each $n\in \mathbb{N}$. Let $\{V_{n}: n\in\mathbb{N}\}$ be a countable neighborhoods base at the point $e$ in $F$, where the closure $\mbox{cl}_{F}V_{n+1}\subset V_{n}$ for each $n\in \mathbb{N}$. For each $n\in \mathbb{N}$, there exists an open subset $W_{n}$ of $G$ such that $V_{n}=W_{n}\cap F$, $W_{n}\subset \overline{U_{n}}$ and $\overline{W_{n+1}}\subset W_{n}$. Put $\gamma=\{W_{n}: n\in \mathbb{N}\}$.
By Lemma~\ref{compact-character}, the family $\gamma$ is a neighborhood base in $G$ at point $e$.
hence the space $G$ is first-countable, and therefore, it is metrizable.
\end{proof}

\begin{corollary}\cite{NT}
The following conditions are equivalent:
\begin{enumerate}
\item Every compact topological group with the $\alpha_{1}$-property is metrizable;

\item Every  locally compact topological group with the $\alpha_{4}$-property is metrizable;

\item $\flat=\omega_{1}$.
\end{enumerate}
\end{corollary}

\begin{question}
Let $G$ be a locally compact rectifiable $\alpha_{4}$-space. Is the space $G$ an $\alpha_{1}$-space in ZFC?
\end{question}

\bigskip

\section{$\alpha_{4}$-rectifiable spaces}
In this section, we first give a new proof of the properties of
Fr$\acute{e}$chet-Urysohn and strongly Fr$\acute{e}$chet-Urysohn are
coincide in rectifiable spaces, which was proved in \cite{LFC2009}.

First, we recall a concept.

($AS$) For any family $\{a_{m, n}: (m, n)\in\mathbb{N}\times\mathbb{N}\}\subset X$ with $\lim_{n}a_{m, n}=a\in X$ for each $m\in\mathbb{N}$, it is possible to choose two strictly increasing sequences $\{i_{l}\}_{l\in \mathbb{N}}\subset \mathbb{N}$ and $\{j_{l}\}_{l\in \mathbb{N}}\subset \mathbb{N}$ such that $\lim_{l}a_{i_{l}, j_{l}}=a$. Obviously, a space with $AS$-property is an $\alpha_{4}$-space.

It is well known that a topological space $X$ is a strongly Fr$\acute{e}$chet-Urysohn space if and only if it is  Fr$\acute{e}$chet-Urysohn and has the double sequence property ($\alpha_{4}$). Therefore, it is sufficient to show that a Fr$\acute{e}$chet-Urysohn rectifiable space has the double sequence property ($\alpha_{4}$). Indeed, we have the following result.

\begin{lemma}\label{t10}
A Fr$\acute{e}$chet-Urysohn Hausdorff rectifiabe space $G$ satisfies $AS$ and hence $\alpha_{4}$ as well.
\end{lemma}

\begin{proof}
Assume that $G$ is a non-discrete space. Let $\{a_{m, n}: (m, n)\in\mathbb{N}\times\mathbb{N}\}\subset X$ with $\lim_{n}a_{m, n}=e$ for each $m\in\mathbb{N}$. Since $G$ is a Fr$\acute{e}$chet-Urysohn non-discrete space, there exists a sequence $\{s_{m}\}m\in \mathbb{N}\subset G$ with $\lim_{m}s_{m}=e$ such that $s_{m}\neq e$ for each $m\in \mathbb{N}$.

Put $z_{m, k}=q(s_{m}, a_{m, k+m})$ if $q(s_{m}, a_{m, k+m})\neq e$, and $z_{m, k}=s_{m}$ if $q(s_{m}, a_{m, k+m})=e$.
Let $M=\{z_{m, k}: (m, k)\in\mathbb{N}\times\mathbb{N}\}$. Obviously, we have $e\not\in M$ since $s_{m}\neq e$ for each $m\in \mathbb{N}$. However, we have $e\in\overline{M}$. Indeed, if $M\cap \{s_{m}: m\in\mathbb{N}\}$ is infinite, then it is easy to see that $e\in\overline{M}$. Therefore, suppose that $M\cap \{s_{m}: m\in\mathbb{N}\}$ is finite. Then there is an open neighborhood $U$ of $e$ such that $U\cap M\cap \{s_{m}: m\in\mathbb{N}\}=\emptyset$. Let $V$ be any open neighborhood of $e$ with $V\subset U$. Hence there is an open neighborhood $W$ of $e$ such that $q(W, W)\subset V$. It follows from $\lim_{m}s_{m}=e$ that there exists an $m\in\mathbb{N}$ such that $s_{m}\in W$. Since $\lim_{n}a_{m, n}=e$, there exists a $k\in \mathbb{N}$ such that $a_{m, k+m}\in W$. Therefore, we have $q(s_{m}, a_{m, k+m})=z_{m, k}\in q(W, W)\subset V\subset U$.

Since $e\in\overline{M}$ and $G$ is Fr$\acute{e}$chet-Urysohn, we can find a sequence $\{(m_{l}, k_{l})\}_{l\in \mathbb{N}}$ in $G$ such that $\lim_{l}z_{m_{l}, k_{l}}=e$.

{\bf Case 1:} The sequence $\{k_{l}\}_{l\in\mathbb{N}}$ is bounded.

Without loss of generality, we may assume that $k_{l}=r$ for each $l\in \mathbb{N}$ for some $r\in \mathbb{N}$.
Since $\lim_{l}z_{m_{l}, k_{l}}=\lim_{l}z_{m_{l}, r}=e$ and $z_{m_{l}, r}\neq e$ for each $l\in \mathbb{N}$, we have $\lim_{l}m_{l}=\infty$. Without loss of generality, suppose that $m_{l}<m_{l+1}$ for each $l\in \mathbb{N}$. Let $N_{1}=\{l\in\mathbb{N}: z_{m_{l}, r}=s_{m_{l}}\}$.

Subcase 1.1: The set $N_{1}$ is infinite.

We denote $N_{1}$ by $\{p_{i}: i\in\mathbb{N}\}$, where $p_{i}<p_{i+1}$ for each $i\in \mathbb{N}$. Then it is easy to see that $q(s_{m_{p_{l}}}, a_{m_{p_{l}}, r+m_{p_{l}}})=e$ for each $l\in \mathbb{N}$. Since $\lim_{l}s_{m_{p_{l}}}=e$, we have $$a_{m_{p_{l}}, r+m_{p_{l}}}=p(s_{m_{p_{l}}}, q(s_{m_{p_{l}}}, a_{m_{p_{l}}, r+m_{p_{l}}}))=p(s_{m_{p_{l}}}, e)=s_{m_{p_{l}}}\rightarrow e\ \mbox{as}\ l\rightarrow\infty.$$ Therefore, we can set $i_{l}=m_{p_{l}}$ and $j_{l}=r+m_{p_{l}}$ for each $l\in \mathbb{N}$. Then we get the strictly increasing sequences $\{i_{l}\}_{l\in \mathbb{N}}$ and $\{j_{l}\}_{l\in \mathbb{N}}$ such that $\lim_{l}a_{i_{l}, j_{l}}=e$.

Subcase 1.2: The set $N_{1}$ is finite.

Let $N_{2}=\{l\in\mathbb{N}: z_{m_{l}, r}\neq s_{m_{l}}\}$. Then $N_{2}$ is infinite. We may denote $N_{2}$ by $\{q_{i}: i\in\mathbb{N}\}$, where $q_{i}<q_{i+1}$ for each $i\in \mathbb{N}$. It follows that $$z_{m_{q_{l}}, k_{q_{l}}}=q(s_{m_{q_{l}}}, a_{m_{q_{l}}, r+m_{q_{l}}})\ \mbox{for each}\ l\in \mathbb{N}.$$ Since $\lim_{l}z_{m_{q_{l}}, k_{q_{l}}}=e$ and $\lim_{l}s_{q_{l}}=e$, we have $$a_{m_{q_{l}}, r+m_{q_{l}}}=p(s_{q_{l}}, q(s_{m_{q_{l}}}, a_{m_{q_{l}}, r+m_{q_{l}}}))=p(s_{q_{l}}, z_{m_{q_{l}}, k_{q_{l}}})\rightarrow p(e, e)=e\ \mbox{as}\ l\rightarrow\infty.$$Therefore, we can set $i_{l}=m_{q_{l}}$ and $j_{l}=r+m_{q_{l}}$ for each $l\in \mathbb{N}$. Then we get the strictly increasing sequences $\{i_{l}\}_{l\in \mathbb{N}}$ and $\{j_{l}\}_{l\in \mathbb{N}}$ such that $\lim_{l}a_{i_{l}, j_{l}}=e$.

{\bf Case 2:} The sequence $\{k_{l}\}_{l\in\mathbb{N}}$ is unbounded.

Without loss of generality, we may assume that $\{k_{l}\}_{l\in \mathbb{N}}$ is a strictly increasing sequence.

Claim: $\lim_{l}m_{l}=\infty$.

If not, we may assume that, for each $l\in \mathbb{N}$, $m_{l}=t$ for some $t\in \mathbb{N}$. Since $\{k_{l}\}_{l\in \mathbb{N}}$ is strictly increasing, we have $\lim_{l}a_{t, t+k_{l}}=e$. It follows from $\lim_{l}z_{t, k_{l}}=e$ that $$a_{t, t+k_{l}}=a_{m_{l}, m_{l}+k_{l}}=p(s_{m_{l}}, z_{m_{l}+k_{l}})=p(s_{t}, z_{t, k_{l}})\rightarrow p(s_{t}, e)=s_{t}\ \mbox{as}\ l\rightarrow\infty.$$ However, $a_{t, t+k_{l}}\rightarrow e$ as $l\rightarrow\infty$. Hence $s_{t}=e$, which is a contradiction.

It follows from Claim that there exists a strictly increasing sequence $\{n_{l}\}_{l\in \mathbb{N}}\subset \mathbb{N}$ such that $m_{n_{i}}<m_{n_{i+1}}$ for each $i\in \mathbb{N}$. Therefore, we can set $i_{l}=n_{l}$ and $j_{l}=m_{n_{l}}+k_{n_{l}}$ for each $l\in \mathbb{N}$. Then we get the strictly increasing sequences $\{i_{l}\}_{l\in \mathbb{N}}$ and $\{j_{l}\}_{l\in \mathbb{N}}$ such that $\lim_{l}a_{i_{l}, j_{l}}=e$.
\end{proof}

It follows from Lemma~\ref{t10}, we have the following theorem, which was proved in \cite{LFC2009}.

\begin{corollary}\label{t15}
A rectifiable space $G$ is Fr$\acute{e}$chet-Urysohn if and only if it is
strongly Fr$\acute{e}$chet-Urysohn.
\end{corollary}

\begin{lemma}\label{l7}
Let $G$ be an $\alpha_{4}$-rectifiable space. If $G$ is a sequential space then $G$ is strongly Fr$\acute{e}$chet-Urysohn.
\end{lemma}

\begin{proof}
It follows from Corollary~\ref{t15} that it suffices to show that $G$ is Fr$\acute{e}$chet-Urysohn. Suppose that $G$ is non-Fr$\acute{e}$chet-Urysohn. Then there exists a subset $A$ of $G$ such that $\hat{\hat{A}}\setminus\hat{A}\neq\emptyset$, where the set $\hat{A}$ is all the limit points of convergent sequences in $A.$ Take a point $x\in\hat{\hat{A}}\setminus\hat{A}$. Without loss of generality, we may assume the $x=e.$

Since $e\in\hat{\hat{A}}$, there exists a sequence $\{x_{n}\}_{n=1}^{\infty}\subset \hat{A}$ such that the sequence
$\{x_{n}\}_{n=1}^{\infty}$ converges to $e$. For each $n\in \mathbb{N}$, there exists a sequence $\{x_{nj}\}_{j=1}^{\infty}\subset A$ such that the sequence
$\{x_{nj}\}_{j=1}^{\infty}$ converges to $x_{n}$. Since $G$ is a rectifiable space, the sequence $\{q(x_{n}, x_{nj})\}_{j=1}^{\infty}$ converges to $q(x_{n}, x_{n})=e$ as $j\rightarrow\infty$. Moreover, since $G$ is an $\alpha_{4}$-rectifiable space, there are an increasing sequence $\{n_{k}\}_{k=1}^{\infty}$ and a sequence $\{j(n_{k})\}_{k=1}^{\infty}$ such that $\{q(x_{n_{k}}, x_{n_{k}j(n_{k})})\}_{k=1}^{\infty}$ converges to $e$. Then we have $$x_{n_{k}j(n_{k})}=p(x_{n_{k}}, q(x_{n_{k}}, x_{n_{k}j(n_{k})}))\rightarrow p(e, e)=e\ \mbox{as}\ k\rightarrow\infty.$$ However, we have $e\not\in\hat{A}$, which is a contradiction.
\end{proof}

It follows from Lemmas~\ref{t10},~\ref{l7} and Corollary~\ref{t15} that we have the following theorem.

\begin{theorem}\label{t12}
Let $G$ be a sequential rectifiable space. Then the following conditions are equivalent:
\begin{enumerate}
\item The space $G$ is an $\alpha_{4}$-space;

\item The space $G$ is an $AS$-space;

\item The space $G$ is Fr$\acute{e}$chet-Urysohn;

\item The space $G$ is strongly Fr$\acute{e}$chet-Urysohn.
\end{enumerate}
\end{theorem}

\begin{corollary}\label{t3}\cite{G1996}
If $G$ is a weakly first-countable rectifiable space, then $G$ is first-countable and hence it is metrizable.
\end{corollary}

\begin{proof}
It is well known that a weakly first-countable space is a sequential $\alpha_4$-space, by Lemma~\ref{l7}, $G$ is a Fr\'echet-Urysohn space, then $G$ is first-countable since a Fr\'echet-Urysohn weakly first-countable space is first-countable. Hence $G$ is metrizable.
\end{proof}
\bigskip

\section{Metrizabilities of rectifiable spaces}
In \cite{LFC2009}, F.C. Lin and R.X. Shen posed the following question:

\begin{question}\cite{LFC2009}\label{q0}
Is every sequential rectifiable space with a point-countable
$k$-network\footnote{ Let $\mathscr{P}$ be a family of subsets of
a space $X$.
The family $\mathscr{P}$ is called a {\it $k$-network} \cite{PO} if
whenever $K$ is a compact subset of $X$ and $K\subset U\in \tau
(X)$, there is a finite subfamily $\mathscr{P}^{\prime}\subset
\mathscr{P}$ such that $K\subset \cup\mathscr{P}^{\prime}\subset U$.} a paracompact space?
\end{question}

In this section, we shall give a partial answer for Question~\ref{q0}. Moreover, we also discuss the metrizability of rectifiable spaces.

Let $X$ be a space and $x\in X.$ The space $X$ has property $P(x, U)$ \cite{SA}
if $U\subset X$, $\{x_{i}: i\in \mathbb{N}\}\subset U$,
$x_{i}\rightarrow x$ as $i\rightarrow\infty$ and $x_{i}\neq x_{j}$
if $i\neq j$, then there is $\Gamma =\{x(n, k): n, k\in
\mathbb{N}\}\subset U$ such that $x(n, k)\rightarrow x_{n}$ as
$k\rightarrow\infty$, $t: \mathbb{N}^{2}\rightarrow \Gamma$ is a
bijection, where $t(n, k)=x(n, k)$ and $\Gamma\cup\{x_{i}: i\in
\mathbb{N}\}\cup\{x\}$ is a closed subset of $X$ homeomorphic to
$S_{2}$.

\begin{lemma}\cite{SA}\label{l0}
A sequential non-Fr$\acute{e}$chet-Urysohn space with a
point-countable $k$-network contains a closed copy of $S_{2}$.
\end{lemma}

\begin{lemma}\cite{LFC2009}\label{l1}
Let $G$ be a rectifiable space. Then $G$ contains a (closed) copy of
$S_{\omega}$ if and only if $G$ has a (closed) copy of $S_{2}$.
\end{lemma}

\begin{lemma}\label{l3}
Let $G$ be a non-Fr$\acute{e}$chet-Urysohn sequential rectifiable
space with point-countable $k$-network. Then for any $x\in G$ and
any open $U\subset G$, $G$ has the property $P(x, U)$.
\end{lemma}

\begin{proof}
By Lemma~\ref{l0}, there exists a closed subset of $G$ homeomorphic
to $S_{2}$. It follows from Lemma~\ref{l1}, $G$ contains a closed
subset homeomorphic to $S_{\omega}$. Let $S_{\omega}=\{y(n, k): n,
k\in \mathbb{N}\}\cup\{e\}$, where $y(n, k)\rightarrow e$ as
$k\rightarrow\infty$ and $y(n, k)\neq y(l, m)$ if $(n, k)\neq (l,
m)$. We may assume that $S_{\omega}$ is a closed subset of $G$. Let
$U$ be open in $G$, $\{x_{i}: i\in \mathbb{N}\}\subset U$,
$x_{i}\rightarrow x$ as $i\rightarrow\infty$ and $x_{i}\neq x_{j}$
if $i\neq j$. For any $i, k\in \mathbb{N}$, put $x(i, k)=p(x_{i},
y(i, k))$. Since $y(i, k)\rightarrow e$ as $k\rightarrow\infty$, it follows that
$x(i, k)\rightarrow p(x_{i}, e)=x_{i}$ as $k\rightarrow\infty$. For
every $i\in \mathbb{N}$, we can choose a $k_{i}\in \mathbb{N}$ such
that $\{x(i, k): i\in \mathbb{N}, k\geq k_{i}\}\subset U$ and $x(i,
k)\neq x(i^{\prime}, k^{\prime})$ if $(i, k)\neq (i^{\prime},
k^{\prime})$ and $k\geq k_{i}, k^{\prime}\geq k_{i^{\prime}}$. Then
$\{x_{i}: i\in \mathbb{N}\}\cup\{x(i, k): i\in \mathbb{N}, k\geq
k_{i}\}\cup\{x\}$ is a closed subset in $G$ and homeomorphic to
$S_{2}$. If not, there exists a sequence $\{x(i_{j}, k_{j})\}_{j=1}^{\infty}$
converging to some point $b\in G$ such that $i_{j}\neq
i_{j^{\prime}}$ if $j\neq j^{\prime}$. Therefore, we have $$y(i_{j},
k_{j})=q(x_{i_{j}}, p(x_{i_{j}},
y(i_{j}, k_{j})))=q(x_{i_{j}}, x(i_{j}, k_{j}))\rightarrow q(x, b)\ \mbox{as}\
j\rightarrow\infty.$$ However, the set $\{y(i_{j},
k_{j}): j\in \mathbb{N}\}$ is closed and discrete in $G$, which is a contradiction.
\end{proof}

\begin{lemma}\cite{SA}\label{l2}
Let $X$ be a sequential space with a point-countable
$k$-network such that for any $x\in X$ and $U\subset X$ the property
$P(x, U)$ holds. Then for any $\alpha <\omega_{1}, x\in X, U\subset
X$ open in $X$ the following property holds:\\
$Q(\alpha , x, U)$: If $\{x_{i}: i\in \mathbb{N}\}\subset U$,
$x_{i}\rightarrow x$ as $i\rightarrow\infty$ then there is $Q\subset
U$ such that $\overline{Q}$ is countable,
$\overline{Q}\setminus\{x\}=U$, $x\in [Q]_{\alpha}$, $x\not\in
[Q]_{\beta}$ for each $\beta <\alpha$.
\end{lemma}

\begin{lemma}\cite{LFC2009}\label{l4}
Let $G$ be a sequential rectifiable space. If $G$ has a
point-countable $k$-network, then $G$ is metrizable if and only if
$G$ contains no closed copy of $S_{2}$.
\end{lemma}

\begin{theorem}
Let $G$ be a sequential rectifiable space with a point-countable $k$-network.
If $so(G)<\omega_{1}$, then $G$ is metrizable.
\end{theorem}

\begin{proof}
Suppose that $so(G)=\alpha$.

Claim: The space $G$ is Fr$\acute{e}$chet-Urysohn.

Suppose not, it follows from Lemmas~\ref{l3} and~\ref{l2} that $G$
has property $Q(\alpha +1, e, G)$. Clearly, since $G$ has the
property $Q(\alpha +1, e, G)$, we have $so(G)\geq \alpha +1>\alpha$, which
is a contradiction.

It follows from the claim that $G$ is a Fr$\acute{e}$chet-Urysohn rectifiable
space, and hence $G$ contains no closed copy of $S_{2}$. Since $G$
is a Fr$\acute{e}$chet-Urysohn rectifiable space with a point-countable
$k$-network, the space $G$ is metrizable by Lemma~\ref{l4}.
\end{proof}

\begin{proposition}
Let $\mathcal{P}$ be a topological property that is productive and preserved by continuous maps. Then
the following are equivalent for a rectifiable space $G$.

(i)Every subset with the property $\mathcal{P}$ of $G$ has countable pseudocharacter.

(ii)Every subset with the property $\mathcal{P}$ of $G$ has regular $G_\delta$-diagonal\footnote{A space $X$ is said to have a {\it regular $G_{\delta}$-diagonal}
if the diagonal $\Delta=\{(x, x): x\in X\}$ can be represented as
the intersection of the closures of a countable family of open
neighborhoods of $\Delta$ in $X\times X$.}.
\end{proposition}

\begin{proof}
(ii) $\to$ (i) obvious.

(i) $\to$ (ii). Let $A$ be a subset of $G$ and have the property $\mathcal{P}$. Since $q: G\times G\to G$ is continuous and the property $\mathcal{P}$ is productive and preserved by continuous maps, then $q(A\times A)=B$ is a subset of $G$ and $B$ has the property $\mathcal{P}$. Then $e\in B$ since $q(x, x)=e$. Therefore, $e$ is a $G_{\delta}$-set of $B$, let $\{U_n: n\in \mathbb{N}\}$ be a family of countable open subsets with $e\in U_n$ and $\overline{U_{n+1}}\subset U_n$. Then $\Delta=\{(x, x): x\in A\}=\bigcap_{n\in\mathbb{N}}q^{-1}(U_n)=\bigcap_{n\in\mathbb{N}}\overline{q^{-1}(U_n)}$. In fact, let $(x, y)\in \bigcap_{n\in\mathbb{N}}q^{-1}(U_n)$. For each $n\in \mathbb{N}$, we have $(x, y)\in q^{-1}(U_n)$, and hence $q(x, y)\in U_n$, which follows that $\pi_2(\varphi(x, y))\in U_n$, $\pi_2(x, y')\in U_n$, where $\varphi(x, y)=(x, y')$ and $y'\in U_n$ for each $n\in \mathbb{N}$. Therefore, $y'=e$. Since $\varphi(x, x)=e$, $\varphi(x, y)=e$ and $\varphi$ is one-to-one, we have $x=y$. Therefore $\Delta=\bigcap_{n\in\mathbb{N}}q^{-1}(U_n)$. Since $\overline{q^{-1}(U_{n+1})}\subset q^{-1}(\overline{U_{n+1}})\subset q^{-1}(U_n)$, then we can see that $A$ has a regular $G_\delta$-diagonal.

\end{proof}

It is well known that a countably compact (compact) space with a $G_\delta$-diagonal is metrizable, we have the following.
\begin{corollary}
The following are equivalent for a rectifiable space $G$.

(i)Every compact (countably compact) subset is first-countable.

(ii)Every compact (countably compact) subset is metrizable.
\end{corollary}

\begin{corollary}
Let $G$ be a rectifiable space of countable pseudocharacter. Then $G$ has a regular $G_\delta$-diagonal.
\end{corollary}
\bigskip

\section{Compactifications of rectifiable spaces}
In this section, we assume that all spaces are Tychonoff.

 Note that a rectifiable space is metrizable if its $\pi$-character is countable \cite{G1996},  by the same proof of \cite[Lemma 2]{Liu2009}, we can prove the following.

\begin{lemma}\label{lemma-1}
Let $G$ be a non-locally compact rectifiable space. If for each $y\in Y=bG\setminus G$, there exists an open neighborhood $U(y)$ of $y$ such that every countably compact subset of $U(y)$ is metrizable and $\pi\chi(U(y))\leq \omega$, then $G$ is metrizable and locally separable.
\end{lemma}

A space $X$ is called having the property (*): if the cardinality of $X$ is Ulam non-measure, then $X$ is weakly HN-complete\footnote{A space $X$ is {\it weakly HN-complete} if the remainder $Z$ of $X$ in the \v{C}ech-Stone compactification $\beta X$ of $X$ is a space of point-countable type.}. A paracompact space has the property (*) since a paracompact space with Ulam non-measurable cardinality is HN-complete \cite{AP1984, E1989}, and hence it is weakly HN-complete.

\begin{proposition}\label{proposition-1}
Let $G$ be a non-locally compact rectifiable space with property (*). If for each $y\in Y=bG\setminus G$, there exists an open neighborhood $U(y)$ of $y$ such that (i) every compact subset of $U(y)$ is a $G_\delta$-subset of $U(y)$; (ii) every countably compact or Lindel\"of $p$-subspace of $U(y)$ is metrizable. Then $G, bG$ are separable and metrizable.
\end{proposition}

\begin{proof}
From condition (ii), we can see that $Y$ is not locally countably compact, otherwise $G$ is closed in $bG$ and is compact.

By \cite[Theorem 3.1]{A2009}, $Y$ is pseudocompact or Lindel\"of.

{\bf Case 1.} The space $Y$ is pseudocompact. Then $Y$ is first-countable since each singleton of $Y$ is a $G_\delta$-set. Since $Y$ is not locally countably compact, the rectifiable space $G$ is locally separable and metrizable by Lemma ~\ref{lemma-1}. $Y$ is Lindel\"of \cite{HI1958} since $G$ is of countable type. Therefore, $Y$ is compact,  and hence $G$ is locally compact, which is a contradiction.

{\bf Case 2.} The space $Y$ is Lindel\"of. Since $Y$ is a space of countable pseudocharacter, it follows that the cardinality of $Y$ is Ulam non-measurable \cite{AP1984}. The space $G$ is not locally compact, then $G$ is nonwhere locally compact since $G$ is homogeneous. It follows that $G$ is a remainder of $Y$, so the cardinality of $G$ is also Ulam non-measurable \cite{AP1984}. Then $G$ is weakly HN-complete. By \cite[Theorem 4]{A20091}, each $G_\delta$-point of $Y$ is a point of bisequentiality of $Y$, it follows that $\pi \chi(Y)\leq \omega$. Therefore, $G$ is locally separable and metrizable by Lemma ~\ref{lemma-1}. We write $G=\bigoplus_{\alpha\in A}G_{\alpha}$, where $G_\alpha$ is a
separable metrizable subset for each $\alpha \in A$. Let
$\eta=\{G_\alpha: \alpha\in A\}$, and let $F$ be the set of all
points of $bG$ at which $\eta$ is not locally finite. Since $\eta$
is discrete in $X$. Then $F\subset bG\setminus G$. It is easy to
see that $F$ is compact, we can find finitely many closed
neighborhoods that satisfy (ii) to cover $F$, hence $F$ is separable
and metrizable, thus $F$ has a countable network. Put
$M=Y\setminus F$. For each point $y\in M$, there is an open neighborhood
$O_y$ satisfying (ii) in $bG$ such that $\overline{O_y}\cap F=\emptyset$. Since
$\eta$ is discrete, the set $\overline{O_y}$ meets at most finitely many
$G_\alpha$. Let $L=\cup\{G_\alpha: G_\alpha\cap \overline{O_y}\neq
\emptyset\}$. Then $L$ is separable metrizable. It follows that $\overline{L}\setminus L$ is a remainder of $L$, and hence it is a Lindel\"of $p$-space by \cite[Theorem 2.1]{A2005}.
$Cl_Y(O_y)\subset \overline{L}\setminus L$, then $Cl_Y(O_y)$ is a Lindel\"of $p$-space, hence it is separable and metrizable and $Y\setminus F$ is locally separable metrizable. Since $F$ is compact, there are finite many $\{U(y_i): i\leq k\}$ that satisfy (i) and cover $F$. Moreover, since each compact subset of $U(y_i) (i\leq k)$ is a $G_\delta$-set, the set $F$ is a $G_\delta$-set in $\cup\{U(y_i): i\leq k\}$. We write $F=\cap V_n$ with $V_n$ open in $Y$ and $Cl_Y(V_{n+1})\subset V_n$. Let $K_1=Y\setminus V_1, K_n=Cl(V_{n-1})\setminus V_n (n>1)$. Since $Y$ is Lindel\"of and $K_n$ is closed in $Y$, each $K_n$ is Lindel\"of and locally separable metrizable. Therefore, $K_n$ has a countable base for each $n$. Since $Y=F\cup (\cup\{K_n: n\in \mathbb{N}\})$, it follows that $Y$ has a countable network. Then $c(Y)\leq \omega$, hence $c(G)\leq \omega$. Since $G$ is a metrizable space with countable Souslin number, the space $G$ is separable and metrizable. It follows that $bG$ is separable and metrizable since $G$ and $Y$ both have countable networks.
\end{proof}

Recall that the space $X$ has a {\it quasi-$G_\delta$-diagonal} provided there is a sequence $\{\mathcal{G}(n): n\in \mathbb{N}\}$ of collections of open sets with property that, given distinct points $x, y\in X$, there is some $n$ with $x\in st(x, \mathcal{G}(n))\subset X\setminus \{y\}$. Obviously, ``$X$ has a $G_\delta$-diagonal" implies ``$X$ has a quasi-$G_\delta$-diagonal".

\smallskip

\begin{theorem}
Let $G$ be a non-locally compact, paracompact rectifiable space, and
$Y=bG\setminus G$ have locally quasi-$G_\delta$-diagonal. Then $G$ and $bG$ are
separable and metrizable.
\end{theorem}

\begin{proof}
 By \cite[Proposition 2.3]{BBL2006}, for $y\in Y$, there exists an open neighborhood $U(y)$ such that each compact subset of $U(y)$ is a $G_\delta$-set and every countably compact subset of $U(y)$ is metrizable. Moreover, every Lindel\"of $p$-subspace of $U(y)$ is metrizable by \cite[Corollary 3.6]{Ho1974}. Then $G$ and $bG$ are separable and metrizable by Proposition ~\ref{proposition-1}
\end{proof}

\begin{corollary}\cite{A2009}
Let $G$ be a non-locally compact, paracompact rectifiable space, and
$Y=bG\setminus G$ have a $G_\delta$-diagonal. Then $G$ and $bG$ are
separable and metrizable.
\end{corollary}

\begin{proposition}\label{proposition-2}
Let $G$ be a non-locally compact rectifiable space. If for each $y\in Y=bG\setminus G$, there exists an open neighborhood $U(y)$ of $y$ such that (i) $\pi\chi (U(y))\leq \omega$; (ii) every countably compact or Lindel\"of $p$-subspace of $U(y)$ is metrizable; (iii) every compact subset of $U(y)$ is a $G_\delta$-set of $U(y)$. Then $G, bG$ are separable and metrizable.
\end{proposition}

\begin{proof}
By Lemma ~\ref{lemma-1}, $G$ is metrizable and locally separable. Similar to the proof of Proposition ~\ref{proposition-1}, $G$ and $bG$ are separable and metrizable.
\end{proof}

A space with point-countable base satisfies (i), (ii) \cite[Corollary 7.11(ii)]{Gr1984} and (iii) in Proposition~\ref{proposition-2}.

\begin{corollary}
Let $G$ be a non-locally compact rectifiable space, and
$Y=bG\setminus G$ have locally point-countable base. Then $G$ and $bG$ are
separable and metrizable.
\end{corollary}

By \cite[Proposition 2.1]{BBL2006} and \cite[Corollary 8.3(ii)]{Gr1984}, a space with a $\delta\theta$-base\footnote{a {\it $\delta\theta$-base} for a space $X$ is a base $\mathcal{B}=\cup\{\mathcal{B}(n): n\geq 1\}$ with the additonal property that $U$ is open and $x\in U$, then there is some $n=n(x, U)$ with properties that (a) some $B\in \mathcal{B}(n)$ has $x\in B\subset U$, and (b) $ord(x, \mathcal{B}(n))\leq \omega$.} satisfies (i), (ii) and (iii) in Proposition~\ref{proposition-2}.

\begin{corollary}
Let $G$ be a non-locally compact rectifiable space, and
$Y=bG\setminus G$ have locally $\delta\theta$-base. Then $G$ and $bG$ are
separable and metrizable.
\end{corollary}

By \cite[Corollary 10.7(ii), Theorem 10.6]{Gr1984}, a $\gamma$-space\footnote{A space $(X, \tau)$ is a {\it $\gamma$-space} if there exists a
function $g:\omega\times X\to \tau$ such that (i) $\{g(n, x): n\in
\omega\}$ is a base at $x$; (ii) for each $n\in \omega$ and $x\in
X$, there exists $m\in \omega$ such that $y\in g(m, x)$ implies
$g(m, y)\subset g(n, x)$.} satisfies (i), (ii) and (iii) in Proposition~\ref{proposition-2}.

\begin{corollary}
Let $G$ be a non-locally compact rectifiable space, and
$Y=bG\setminus G$ be a locally $\gamma$-space. Then $G$ and $bG$ are
separable and metrizable.
\end{corollary}

{\bf Acknowledgements}. We wish to thank
the reviewers for the detailed list of corrections, suggestions to the paper, and all her/his efforts
in order to improve the paper.

\end{document}